\documentclass[journal]{IEEEtran}                                                          % paper
\IEEEoverridecommandlockouts

\usepackage{graphicx}
\usepackage{amssymb}
\usepackage{algorithmicx, algpseudocode}
\usepackage{comment}
\usepackage{algorithm}
\usepackage{algpseudocode}
\usepackage{amsthm}
\usepackage{amsmath} % assumes amsmath package installed
\usepackage{algorithmicx}
\newcommand{\real}{{\mathbb{R}}}
\newcommand{\reals}{{\mathbb{R}}}
\newcommand{\integer}{{\mathbb{Z}}}

\newcommand{\realpp}{{\mathbb{R}}_{\ge 0}}
\newcommand{\Lnorm}{\left\|} \newcommand{\Rnorm}{\right\|}
 
\newcommand{\eps}{\epsilon}

\renewcommand{\baselinestretch}{.985}

\newcommand{\vect}[1]{\boldsymbol{\mathbf{#1}}}

\newcommand{\dvect}[1]{\dot{\vect{#1}}}

 \newcommand{\boxend}{\hfill \ensuremath{\Box}}

\newcommand{\oprocendsymbol}{\hbox{$\bullet$}}
\newcommand{\oprocend}{\relax\ifmmode\else\unskip\hfill\fi\oprocendsymbol}

\newtheorem{thm}{Theorem}[section]

\newtheorem{rem}{Remark}[section]

\newtheorem{lem}{Lemma}[section]

\newtheorem{assump}{Assumption}



\begin{document}

\title{First-Order Dynamic Optimization for\\ Streaming Convex Costs}

\author{Mohammadreza Rostami, Hossein Moradian, and Solmaz S. Kia, \emph{Senior member, IEEE} %
  \thanks{This work was supported by NSF CAREER award ECCS 1653838. The authors are with the Department of Mechanical and Aerospace Engineering, University of California Irvine, Irvine, CA 92697,  
    {\tt\small \{mrostam2,hmoradia,solmaz\}@uci.edu.}}%
}

\markboth{}%
{}
\maketitle

\begin{abstract}
This paper proposes a set of novel optimization algorithms for solving a class of  convex optimization problems with time-varying streaming cost function. We develop an approach to track the  optimal solution with a bounded error. Unlike the existing results, our algorithm is executed only by using the first-order derivatives of the cost function which makes it computationally efficient for optimization with time-varying cost function. We compare our algorithms to the gradient descent algorithm and show why gradient descent is not an effective solution for optimization problems with time-varying cost. Several examples including solving a model predictive control problem cast  as a convex optimization problem with a streaming time-varying cost function demonstrate our results.

\end{abstract}

\IEEEpeerreviewmaketitle

\section{Introduction}

There is growing attention in the broad domain of optimization and learning
 to the problems where traditional optimization technique cannot provide efficient solutions at time scales that match the pace of streaming data due to limitation on computational and communication resources~\cite{ling2013decentralized,MF-SP-AR:17,SR-WR:17,AS-AM-AK-GL-AR:16,ED-AS-SB-LM:19,BH-YZ-ZM-WR:20,esteki2022distributed,MR-SSK:23,BM-EM-DR-SH-BAA:17,YR-XZ-SCL-CJW:21}.
Power grids, networked autonomous systems, real-time data processing, learning methods, and data-driven control systems are among many applications that can significantly benefit from computationally efficient optimization algorithms where data points are processed without central storage. This paper aims to extend the current knowledge on such optimization algorithms with time-varying cost~function.

Consider a class of convex optimization programs where the objective is  time-varying. Formally, for a variable
$\vect{x}\in\real^n$ 
we define a time-varying objective function $f:\real^n\times\real_{\geq 0}\rightarrow\real$
taking values $f(\vect{x},t)$. 
The optimization problem is given as
\begin{align}\label{eq::opt}
  \vect{x}^{\star}(t) &= \arg\underset{{\vect{x}\in
    \reals^n}}{\min} 
  \,\,f(\vect{x},t),\quad t\in\real_{\geq0}.
\end{align}

See Fig.~\ref{fig:illustration_page1} for an illustration of the components of~\eqref{eq::opt}. Suppose at any time $t\in\realpp$ and any finite $\vect{x}\in\real^n$, we have $|f(\vect{x},t)|<\infty$. Moreover, at any $t\in\realpp$, optimization problem~\eqref{eq::opt} is solvable and its minimum value is finite, i.e., $|f(\vect{x}^\star(t),t)|=\mathsf{f}_t^\star<\infty$. In what follows, we denote $\vect{x}^\star(t)$ by $\vect{x}^\star_t$. Notice that given $t\mapsto\vect{x}^\star_t$ as the optimum solution of~\eqref{eq::opt}  we have  
\begin{align}\label{eq::KKT_timeVar}
 \nabla_{\vect{x}}f(\vect{x}_t^\star,t)=\vect{0},~~~\quad t\in\realpp,
 \end{align}
 which is a sufficient condition for $\vect{x}^\star_t$ being the optimal trajectory~\cite{SR-WR:17}.

 \begin{figure}[t]
  \centering
    \includegraphics[scale=0.09]{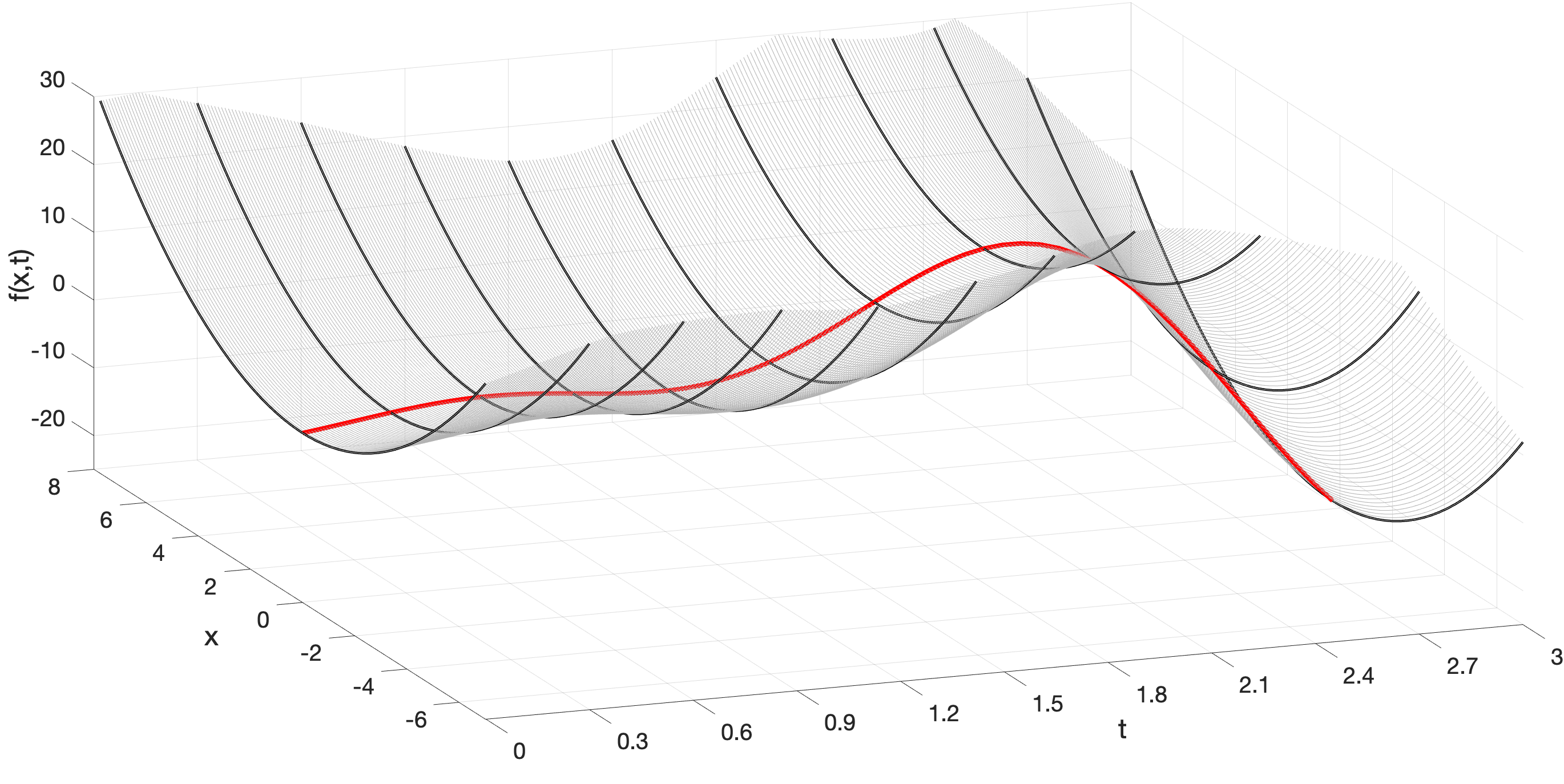}
      \caption{\small{A time-varying $f(\vect{x},t)$ vs. $\vect{x}$ and $t$ (gray plot) and the trajectory of $f(\vect{x}^\star(t),t)$ vs. $\vect{x}^\star(t)$ and $t$ (red curve).} }\label{fig:illustration_page1}
\end{figure}

A trivial way to solve~\eqref{eq::opt} is to sample the problem at particular
times, $t_k\in\real_{\geq0}$, and solve the corresponding
sequence of optimization problems assuming $f(\vect{x},t)=f(\vect{x}(t_k))$ for each time interval $t\in[t_k,t_{k+1})$. However, this 
implementation is expected to result in steady-state tracking errors.
This tracking error can be significant if optimal solution is drifting away from $\vect{x}$ at a rapid pace in each interval.

%\Mohammadreza{
%The optimal solution in \eqref{eq::opt} can be seen as a sequence of optimization problems. In most of the cases, a solver cannot obtain the optimal solution by performing multiple gradient descent steps before going to the next timestep because based on the complexity of the problem that we are dealing with (e.g. high dimensional problem), a solver may not have enough time to solve each problem (e.g. executing multiple gradient descent steps) utterly before moving to the next timestep (e.g. next dataset). Hence, we need to design both efficient and fast algorithms to solve each problem partly before approaching to the next timestep.}

Recently, \cite{PCV-SJ-FB:22} and \cite{DHN-LVT-TK-SJJ:18} have  studied solving the time-varying convex optimization problems from a contraction theory perspective and has provided tracking error bounds between any solution trajectory and the equilibrium trajectory for continuous-time time-varying primal-dual dynamics. In the recent work~\cite{AD-VC-AG-GR-FB:23}, it is shown that for any contracting dynamics dependent on parameters, the tracking error is uniformly upper-bounded in terms of the contraction rate, the Lipschitz constant in which the parameter is involved, and the rate of change of the parameter.
 
On the other hand, for differentiable costs, by observing that taking the derivative of~\eqref{eq::KKT_timeVar} leads to  
 \begin{align}\label{eq::x_star_trajectory}
 \dvect{x}_t^\star = -\nabla_{\vect{xx}}^{-1} f(\vect{x}_t^\star,t) \nabla_{\vect{x}t} f(\vect{x}_t^\star,t),\quad t\in\realpp.
 \end{align}
%alternative approach is to use a  gradient vanishing technique by
%reducing the gradient of the cost function at every time instant with a decay rate. Such \blue{approaches are} proposed in

Literature like~\cite{MF-SP-AR:17,SR-WR:17,ECH-RMW:15,YD-JL-MA:21} proposed continuous-time algorithms that aim to converge to trajectories satisfying~\eqref{eq::x_star_trajectory} asymptotically starting from any initial conditions. % algorithms that which guarantees exponential convergence in continuous-time framework to $\vect{x}^\star_t$, the optimum solution of the optimization problem~\eqref{eq::opt}. 
Following the same methodology, some prediction-correction approaches are  proposed in~\cite{AS-AM-AK-GL-AR:16,ED-AS-SB-LM:19,NB-AS-RC:20,NB:21} in the discrete-time framework. However, these existing algorithms use the second-order derivatives of the cost to achieve convergence. Similar to~\eqref{eq::x_star_trajectory}, they also require the inverse of the Hessian, which adds to the computational cost of these algorithms.

Although elegant, all these approaches suffer from high computational complexity as the algorithms require second-order derivatives of cost function as well as computing the inverse of the Hessian, which is not efficient for high dimensional problems. Requiring to compute the inverse of Hessian also limits the use of these algorithms for non-convex optimization problems where the inverse may not exists. The use of Hessian has also proven to be restrictive in the design of distributed optimization algorithms that are inspired by these second-order algorithms. For example, the algorithms in~\cite{SR-WR:17} and~\cite{BH-YZ-ZM-WR:20} require that Hessians of all the local objective functions be identical.

In this paper, we propose a  solution for the optimization problem~\eqref{eq::opt} in a discrete-time framework, which guarantees convergence to the neighborhood of the optimal solution. Our contribution is that our proposed algorithms are executed by using only the first-order derivatives of the cost function, which reduces the computational costs compared to the Hessian-based algorithms and can also be used in solving non-convex optimization problems. Our first-order algorithms are also designed in discrete-time.

\section{Preliminaries}\label{sec::prelim}
This section defines our notation and states our standing assumptions through out the paper. 
\subsection{Notation}

The set of real numbers is $\real$. For a vector $\vect{x}\in\reals^n$ we define  the Euclidean and infinity norms by, respectively,  $\|\vect{x}\|\!=\!\sqrt{\vect{x}^\top\vect{x}}$ and $\|\vect{x}\|_\infty\!=\!\max{|x_i|}_{i=1}^n$.  
%For $
%\vect{A}\!\in\real^{n\times m}$ and $\vect{B}\in\real^{p\times q}$, we
%let $ \vect{A}\!\kronecker\!\vect{B}$ denote their Kronecker
%product. For vectors $\vect{u}_1,\cdots,\vect{u}_m$, we let
%$\vect{u} = (\vect{u}_1,\cdots,\vect{u}_m)$ represent the aggregated
%vector. 

The partial derivatives of a function $f(\vect{x},t):\real^n\times\real_{>0}\rightarrow\real^n$ with respect to $\vect{x}\in\real^n$ and $t\in\real_{>0}$ are given by $\nabla_{\vect{x}}f(\vect{x},t)$ (referred to hereafter as `gradient') and $\nabla_{t}f(\vect{x},t)$. Also, $\nabla_{\vect{x}t}f(\vect{x},t)$ and $\nabla_{\vect{x}\vect{x}}f(\vect{x},t)$ are respectively  the partial derivatives of $\nabla_{\vect{x}}f(\vect{x},t)$ with respect to $t$ and $\vect{x}$. 
A differentiable function $f: \reals^d\to\reals$ is \emph{$m$-strongly convex} ($m\in\real_{>0}$) in $\real^d$ if and only if
\begin{equation*}
  (\vect{z}-\vect{x})\!^\top\!(
  \nabla f(\vect{z})-\nabla f(\vect{x}))\!\geq
  m\|\vect{z}-\vect{x}\|^2, ~~\forall
  \vect{x},\vect{z}\in \!\real^d,~\vect{x}\!\neq\!\vect{z}. 
\end{equation*}
For twice differentiable function $f$ the $m$-strong convexity is equivalent to $\nabla^2 f(\vect{x})\succeq m \vect{I},~\forall
  ~\vect{x}\in \real^d. $
The gradient of a differentiable function $f:\real^d\to\real$ is globally Lipschitz with constant $M\in\real_{>0}$ (hereafter referred to simply as \emph{$M$-Lipschitz}) 
if and only if
\begin{align}\label{eq::Lipsch}
  \Lnorm \nabla \vect{f}(\vect{x})-\nabla \vect{f}(\vect{y})\Rnorm\leq M\Lnorm
  \vect{x}-\vect{y}\Rnorm, \quad \forall\, \vect{x},\vect{y}\in \real^d.
\end{align}
For twice differentiable functions, condition~\eqref{eq::Lipsch} is equivalent to $\nabla^2 f(\vect{x})\preceq 
  M \vect{I},~\forall
  ~\vect{x}\in \real^d$~\cite{YN:18}. A differentiable $m$-strongly function with a globally $M$-Lipschitz gradient for all $\vect{x},\vect{y}\in\real^d$ also satisfies~\cite{YN:18}
  \begin{subequations}
  \begin{align}
   &   \nabla f(\vect{x})\!^\top\!(\vect{y}-\vect{x})+\frac{m}{2}\|\vect{y}-\vect{x}\|^2\leq f(\vect{y})-f(\vect{x})\nonumber\\
      &\qquad\quad\quad \leq \nabla f(\vect{x})^\top(\vect{y}-\vect{x})+\frac{M}{2}\|\vect{y}-\vect{x}\|^2,\label{eq::convex_lipsc}\\
 &     \frac{1}{2M}\|\nabla f(\vect{x})\|^2\leq f(\vect{x})-f^\star \leq \frac{1}{2m}\|\nabla f(\vect{x})\|^2,\label{eq::dist_optim_cost}
  \end{align}
  \end{subequations}
  where $f^\star$ is the minimum value of $f$.

\subsection{Assumptions on the streaming cost}
We start by some well-posedness conditions on the steaming cost. 
\begin{assump}[Well-posedness conditions]\label{asm:str_convexity}
{\rm 
At any time $t\in\realpp$ and any finite $\vect{x}\in\real^n$, we have $|f(\vect{x},t)|<\infty$. Moreover, at any $t\in\realpp$, optimization problem~\eqref{eq::opt} is solvable and its minimum value is finite, i.e., $|f(\vect{x}^\star(t),t)|=\mathsf{f}_t^\star<\infty$.
} \boxend
\end{assump}
The optimality condition~\eqref{eq::KKT_timeVar} characterizes the solution of the optimization problem~\eqref{eq::opt}. Under the following assumption, the solution $t\mapsto\vect{x}_t^\star$ is unique~\cite{ALD-RTR:09}.
\begin{assump}[Conditions for the existence of a unique solution]\label{asm:str_convexity}
{\rm 
The cost function $f(\vect{x},t):\real^n\times\realpp\rightarrow\real$ is
twice continuously differentiable with respect to $\vect{x}$ and continuously differentiable with respect to $t$ and globally Lipschitz in $t$. 
Moreover,
the cost function is $m$-strongly convex and has $M$-Lipschitz gradient in $\vect{x}$, i.e.,   $$m\,\vect{I}_n\preceq \nabla_{\vect{x}\vect{x}}f(\vect{x},t)\preceq M\,\vect{I}_n,\quad  t\in\realpp.$$ %for any %t\in\realpp$.
} \boxend
\end{assump}

 Beside Assumption~\ref{asm:str_convexity}, to establish our convergence, we place some other conditions on the cost as stated in Assumption~\eqref{asm:bound_dfstar} below. Note that similar Assumptions are also used in the existing second-order time-varying optimization algorithms. 
\begin{assump}[Smoothness of the cost function]\label{asm:bound_dfstar}
{\rm 
 The function  $f(\vect{x},t):\real^n\times\real\rightarrow\real$ is  continuously differentiable and sufficiently smooth. Specifically,
there exists a bound on the derivative of $f(\vect{x},t)$ as 
\begin{align*}
&\|\nabla_{\vect{x}t}f(\vect{x},t)\|\leq K_2, \\
\qquad &|\nabla_{t}f(\vect{x},t)|\leq K_1,\quad  |\nabla_{tt}f(\vect{x},t)|\leq K_3
\end{align*}
for any $\vect{x}\in\real^n$, $t\in\real_{\geq0}$. }\boxend
\end{assump}
By virtue of \cite[Lemma 3.3]{HKK:02}, Assumption~\ref{asm:bound_dfstar} leads to cost function $f(\vect{x},t)$ and its first derivatives $\nabla_{t}f(\vect{x},t)$ and $\nabla_{\vect{x}}f(\vect{x},t)$ being globally Lipschitz in $t\in\real_{\geq0}$ with respective constants $K_1$, $K_2$ and $K_3$. 

In what follows, we let
\begin{align}
    f^\star(t)=f(\vect{x}^\star(t),t),
\end{align}
be the optimal cost value at any $t\in\real_{\geq0}$. 

\begin{lem}[Bound on the  cost difference of the optimizer]\label{bound_traj_opt}{\rm
Consider the optimization problem~\eqref{eq::opt} under Assumptions~\ref{asm:str_convexity} and~\ref{asm:bound_dfstar}. Then,
\begin{align}\label{eq::bound_optimal_trajectory}
    |f^\star(t_{k+1})-f^\star(t_{k})|\leq \psi,
\end{align}
where $\psi=\delta (K_1 + \frac{\delta}{2}K_3) + \frac{K^2_2 \delta^2}{2m} (\frac{M\delta}{m} + 2)$,
with $\delta=t_{k+1}-t_k\in\real_{>0}$ being the sampling timestep of the optimal cost function across time.
}
%{\rm Given Assumption~\ref{asm:str_convexity} and~\ref{asm:bound_dfstar}, the following equation holds for the cost change of the optimal solution defined in~\eqref{eq::opt}.
%\begin{align}\label{eq::bound_optimal_trajectory}
%    \|f^\star(t_{k+1})-f^\star(t_{k})\|\leq\frac{K_1K_2\,\delta}{m}.
%\end{align}
%}
\end{lem}

\begin{proof} Given Assumptions~\ref{asm:str_convexity} and~\ref{asm:bound_dfstar}, it follows from~\cite[Theorem 2F.10]{ALD-RTR:09}, that the trajectory of $\vect{x}^\star(t)$ satisfy 
\begin{align}\label{eq_bound_opt_traj}
    \|\vect{x}^\star_{k+1}-\vect{x}^\star_{k}\| \leq \frac{1}{m} \|\nabla_{\vect{x}t}f(\vect{x}_k,t_k)\| (t_{k+1}-t_k)\leq \frac{K_2\delta}{m}.~
\end{align}
Using Taylor series expansion~\cite{DPB:99} we can write
\begin{align}\label{eq::star_bound}
 &f(\vect{x}^\star_{k+1},t_{k+1})-f(\vect{x}^\star_{k},t_{k})=\nabla_{\vect{x}} f(\vect{x}^\star_{k},t_k)^\top(\vect{x}^\star_{k+1}-\vect{x}^\star_{k})\nonumber\\&+\nabla_{t} f(\vect{x}_{k},t_k)(t_{k+1}-t_k)+\frac{(t_{k+1}-t_k)^2}{2}\nabla_{tt} f(\zeta,\theta)+\nonumber\\
 &+\frac{1}{2}(\vect{x}^\star_{k+1}-\vect{x}^\star_{k})^\top \nabla_{\vect{xx}}f(\zeta,\theta)(\vect{x}^\star_{k+1}-\vect{x}^\star_{k})+\nonumber\\
 &\quad\nabla_{\vect{x}t} f(\zeta,\theta)^\top(\vect{x}^\star_{k+1}-\vect{x}^\star_{k})(t_{k+1}-t_k),
\end{align}
where $\zeta\in[\vect{x}^\star_k,\vect{x}^\star_{k+1})$ component-wise and $\theta\in[t_k,t_{k+1})$. The proof then follows from upper-bounding the right hand side of~\eqref{eq::star_bound} using the bounds in Assumptions~\eqref{asm:str_convexity} and \eqref{asm:bound_dfstar} and the bounds established in~\eqref{eq_bound_opt_traj} and~\eqref{eq::star_bound}.
\end{proof}

\section{Objective statement}\label{sec::Prob_formu}
%  \begin{figure}[t]
%   \centering
%    % \includegraphics[scale=0.83]{Fig/time_varying2.eps}
%     \includegraphics[scale=0.18]{num1_exp.png}
%       \caption{Illustrative figure that showcases the role of the prediction step (line 3) inside algorithm \ref{Alg::1} and \ref{Alg::2} when \mbox{$\nabla_t f(\vect{x}_k,t_k) \geq 0$}.}
%       \label{fig:example-1}
% \end{figure}

%  \begin{figure}[t]
%   \centering
%    % \includegraphics[scale=0.83]{Fig/time_varying2.eps}
%     \includegraphics[scale=0.18]{Fig/num2_exp.png}
%       \caption{Illustrative figure that showcases the role of the prediction step (line 3) inside algorithm \ref{Alg::1} and \ref{Alg::2} when \mbox{$\nabla_t f(\vect{x}_k,t_k) < 0$}.}
%       \label{fig:example-1}
% \end{figure}

 \begin{figure}[t]
  \centering
    \includegraphics[scale=0.29]{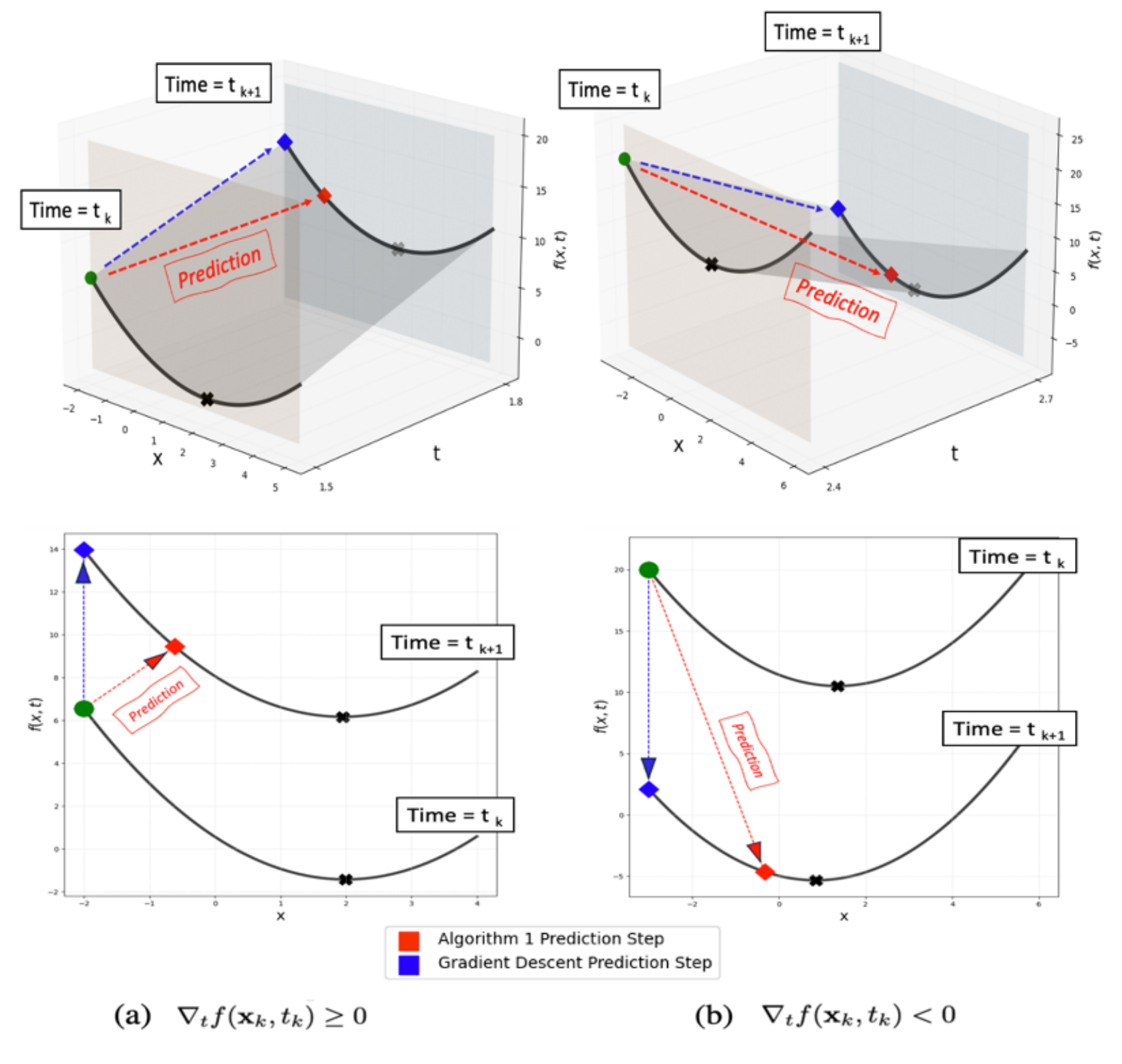}
      \caption{\small{An example case that demonstrates the role of the prediction step (line 3) of Algorithm \ref{Alg::1}. As we can see in this example, for both cases of $\nabla_t f(\vect{x}_k,t_k) \geq 0 $ (plots in the left column) and $\nabla_t f(\vect{x}_k,t_k) < 0 $ (plots in the right column) the statement of Lemma~\ref{lem::grad_des_vs_Alg1} holds, i.e., $ f^-_1(t_{k+1})\leq f^-_{\text{g}}(t_{k+1})$.}}
      \label{fig:example-1}
\end{figure}

 For a first-order solver for problem~\eqref{eq::opt}, some work such as~\cite{AYP:05} investigated use of the conventional gradient descent~algorithm 
\begin{align}\label{eq::grad_desc}
    &\vect{x}_{k+1}=\vect{x}_k-\alpha\, \nabla_{\vect{x}} f(\vect{x}_k,t_{k+1}).
\end{align}
Considering the first-order approximation of the cost at $t_{k+1}$,
\begin{align}
 f(\vect{x}_{k+1},t_{k+1})\approx& f(\vect{x}_{k},t_{k+1})+\nabla_{\vect{x}}f(\vect{x}_{k},t_{k+1})^\top(\vect{x}_{k+1}-\vect{x}_{k}),
\end{align}
the gradient decent algorithm~\eqref{eq::grad_desc} certainly results in function reduction at each $t_{k+1}$. On the other hand, considering the first-order approximation across time from $t_k$ to $t_{k+1}$, 
\begin{align}\label{eq::taylor_time}
 f(\vect{x}_{k},t_{k+1})\approx& f(\vect{x}_{k},t_{k})+\nabla_{t}f(\vect{x}_{k},t_k)^\top\delta .
\end{align}
where $\delta=t_{k+1}-t_k$, we see that the gradient descent algorithm~\eqref{eq::grad_desc} is oblivious to $\nabla_{t}f(\vect{x}_{k},t_k)^\top(t_{k+1}-t_k)$ and how cost is changing across time. Iterative optimization algorithms for unconstrained problems are driven by successive descent objective. However, if $\nabla_{t}f(\vect{x}_{k},t_k)^\top>0$, we have  $f(\vect{x}_{k},t_{k+1})>f(\vect{x}_{k},t_{k})$ and thus the function reduction at $t_{k+1}$ after taking the gradient descent algorithm~\eqref{eq::grad_desc}  is not the same as if $f(\vect{x}_{k},t_{k+1})\approx f(\vect{x}_{k},t_{k})$ (in first-order sense). Thus, one can anticipate that the gradient descent algorithm will result in poor tracking~performance during periods of time that $\nabla_{t}f(\vect{x}(t),t)>0$. 

Let us re-write the gradient decent algorithm~\eqref{eq::grad_desc} as 
\begin{subequations}
\begin{align}\label{eq::grad_desc_alt}
 \vect{x}^-_{k+1}&=\vect{x}_{k},\\
   \vect{x}_{k+1}&=\vect{x}_{k+1}^--\alpha\, \nabla_{\vect{x}} f(\vect{x}_{k+1}^-,t_{k+1}).
\end{align}
\end{subequations}
We refer to $\vect{x}^-_{k+1}$ as \emph{predicted} decision variable at time $t_{k+1}$ and to $\vect{x}_{k+1}$ as the \emph{updated} decision variable at $t_{k+1}$. We expect that at each time $t_{k+1}$, the updated decision variable gets close to $\vect{x}^\star(t_{k+1})$ by virtue of function descent. In the subsequent sections, we set to design alternative algorithms which consider the variation of the cost across time and employ prediction rules that will result in a better tracking performance than that of the gradient descent algorithm.

\section{First-order algorithms}\label{first_order}
In this section, we propose two classes of first-order algorithms to solve optimization problem~\eqref{eq::opt}, using prediction and update steps that take advantage of the first-order term $\nabla_{t} f(\vect{x}{k},t_k)$ to improve convergence performance. Algorithm~\ref{Alg::1} is our first proposed algorithm. The structure of this algorithm consists of two steps: the \emph{prediction step} changes the local state based on the rate of change of the cost function with respect to time. The subsequent \emph{update step} is a gradient descent step at freeze time $t_{k+1}$. The following Lemma reveals the advantage of the prediction step of Algorithm~\ref{Alg::1} over the gradient descent algorithm~\eqref{eq::grad_desc_alt}, which lacks prediction oversight.

\begin{lem}\label{lem::grad_des_vs_Alg1}
Consider the gradient descent algorithm~\eqref{eq::grad_desc_alt} and Algorithm~\ref{Alg::1}. Let $\delta=(t_{k+1}-t_k)$ be the same for both algorithms. Let $f^-_{\text{g}}(t_{k+1})$ and $f^-_1(t_{k+1})$ be, respectively, the function value  of the gradient descent algorithm and Algorithm~\ref{Alg::1} after the prediction step. Suppose $\vect{x}_k$ for both algorithms is the same. Then, for any $\eps\in\real_{\geq0}$, we have $ f^-_1(t_{k+1})\leq f^-_{\text{g}}(t_{k+1})$ in first-order approximate sense.
\end{lem}

\begin{proof}
The first-order approximation of $f(\vect{x}^-_{k+1},t_{k+1})$ is
\begin{align*}
f(\vect{x}^-_{k+1},t_{k+1})\approx &f(\vect{x}_k,t_{k})\,+\nabla_{\vect{x}} f(\vect{x}_k,t_k)(\vect{x}^-_{k+1}-\vect{x}_k)\\
&+\nabla_t f(\vect{x}_k,t_k)(t_{k+1}-t_k).
\end{align*}
For gradient descent algorithm by substitution we obtain $f^-_{\text{g}}(t_{k+1})\approx f(\vect{x}_k,t_{k})+\delta\,\nabla_t f(\vect{x}_k,t_k)$. For Algorithm~\ref{Alg::1}, for $\|\nabla_{\vect{x}}f(\vect{x}_k,t_k)\|<\eps$, we have $f^-_{\text{g}}(t_{k+1})=f^-_1(t_{k+1})\approx f(\vect{x}_k,t_{k})+\delta\,\nabla_t f(\vect{x}_k,t_k)$. For $\|\nabla_{\vect{x}}f(\vect{x}_k,t_k)\|\geq \eps$, on the other hand we have 
\begin{align*}
f^-_1(t_{k+1})\approx\,& f(\vect{x}_k,t_{k})+\nabla_{\vect{x}} f(\vect{x}_k,t_k)(\vect{x}^-_{k+1}-\vect{x}_k)\\
&+\delta \,\nabla_t f(\vect{x}_k,t_k) \\
=\,&f(\vect{x}_k,t_{k})+\delta\,(\nabla_t f(\vect{x}_k,t_k)-|\nabla_t f(\vect{x}_k,t_k)|),
\end{align*}
confirming $f^-_1(t_{k+1})\leq f^-_{\text{g}}(t_{k+1})$, and completing the proof.
\end{proof}
See Fig.~\ref{fig:illustration} for a trajectories generated by gradient descent algorithm and Algorithm~\ref{Alg::1}, which shows Algorithm~\ref{Alg::1} results in a lower tracking error. Next, we present the convergence analysis of Algorithm~\ref{Alg::1}.
\medskip
\begin{thm}[Convergence analysis of Algorithms~\ref{Alg::1}]\label{thm::main1}
{\rm Let %$\vect{x}^\star_{t_k}$ be defined as in~\eqref{eq::opt}  and let
Assumptions~\ref{asm:str_convexity} and~\ref{asm:bound_dfstar}  hold. Then,  Algorithm~\ref{Alg::1} converges to the neighborhood of the optimum solution of~\eqref{eq::opt} with the following upper bound
\begin{align}\label{eq::bound_on_solu_1}
&f(\vect{x}_{k+1},t_{k+1})-f^*(t_{k+1}) \leq \frac{\big(1-(1-2\kappa\alpha m)^k\big)}{4\kappa^2\alpha^2 m} \psi\nonumber\\&
\quad\qquad +\frac{\big(1-(1-2\kappa\alpha m)^k\big)}{4\kappa^2\alpha^2 m^2}\max(\gamma\delta^2,%\frac{\eps^2}{2m})
\mu\delta)\nonumber\\
&\quad \qquad+
(1-2\kappa\alpha m)^k(f(\vect{x}_{0},t_{0})-f^*(t_{0})),
\end{align}
where $\psi = \delta (K_1 + \frac{\delta}{2}K_3) + \frac{K^2_2 \delta^2}{2m} (\frac{M\delta}{m} + 2)$, $\delta=t_{k+1}-t_k$, $\gamma= \frac{2K_1}{\delta} + \frac{M}{2\epsilon^2}K_1^2+\frac{1}{2}K_3 + \frac{1}{\epsilon}K_1K_2$, $\mu=K_1+\frac{\delta}{2}K_3$ and $\kappa=(1-\alpha M/2)$, provided that $0<\alpha\leq \frac{1}{2M}$.
}
\end{thm}

\begin{algorithm}[t]
\caption{$\eps-$exact Time-Varying Optimization with $\nabla_t  f(\vect{x}_k,t_k)$}\label{Alg::1}
%\scriptsize
\begin{algorithmic}[1]
%\Procedure{\sf{TVOpt}}{$\mathcal{P}^1,\cdots,\mathcal{P}^M,M$}
\State {$\mathbf{\textbf{Initialization}}\,\, \vect{x}_0\in\real^n,\,\,\, \eps, \delta\in\real_{>0}, f(\vect{x}_0,t_0)\in\real $, }
\If{$\|\nabla_{\vect{x}}f(\vect{x}_k,t_k)\|\geq \eps$}
\State {$\vect{x}^-_{k+1}=\vect{x}_{k}-\delta\frac{ |\nabla_t f(\vect{x}_k,t_{k})|}{\|\nabla_{\vect{x}}f(\vect{x}_k,t_k)\|^2}\nabla_{\vect{x}}f(\vect{x}_k,t_k)$ }
\Else
\State{\textbf{Set} $\vect{x}^-_{k+1}=\vect{x}_{k}$}
\EndIf 
\State {\textbf{Update} $f(\vect{x}^-_{k+1},t_{k+1})$}
\State {$\vect{x}_{k+1}=\vect{x}^{-}_{k+1}-\alpha\nabla_{\vect{x}}f(\vect{x}^{-}_{k+1},t_{k+1})$}. 
\end{algorithmic}
\label{Alg::1}
\end{algorithm} 

\begin{proof}
Invoking the Taylor series expansion~\cite{DPB:99}, the prediction step of Algorithm~\ref{Alg::1}  gives 
\begin{align}\label{eq::Taylor}
 &f(\vect{x}^-_{k+1},t_{k+1})= f(\vect{x}_{k},t_{k})+\nabla_{\vect{x}} f(\vect{x}_{k},t_k)^\top(\vect{x}^-_{k+1}-\vect{x}_{k})\nonumber\\&+\nabla_{t} f(\vect{x}_{k},t_k)(t_{k+1}-t_k)+\frac{(t_{k+1}-t_k)^2}{2}\nabla_{tt} f(\zeta,\theta)+\nonumber\\&+\frac{1}{2}(\vect{x}^-_{k+1}-\vect{x}_{k})^\top \nabla_{\vect{xx}}f(\zeta,\theta)(\vect{x}^-_{k+1}-\vect{x}_{k})+\nonumber\\&\quad\nabla_{\vect{x}t} f(\zeta,\theta)^\top(\vect{x}^-_{k+1}-\vect{x}_{k})(t_{k+1}-t_k),
\end{align}
where $\zeta\in[\vect{x}_k,\vect{x}_{k+1})$ component-wise and $\theta\in[t_k,t_{k+1})$.
If $\|\nabla_{\vect{x}}f(\vect{x}_k,t_k)\|< \eps$, we have $\vect{x}_{k+1}^-=\vect{x}_k$. Then, given Assumption~\ref{asm:bound_dfstar} we obtain
\begin{align}\label{eq::grad_less_eps}
|f(\vect{x}^-_{k+1},t_{k+1})- f(\vect{x}_{k},t_{k})|\leq \mu\,\delta,\end{align}
where $\mu$ is given in the statement. If $\|\nabla_{\vect{x}}f(\vect{x}_k,t_k)\|\geq \eps$, substituting for $\vect{x}_{k+1}^-$ from Algorithm~\ref{Alg::1} (step $3$) gives
\begin{align*}
&f(\vect{x}^-_{k+1},t_{k+1})- f(\vect{x}_{k},t_{k})=\frac{\delta^2}{2}\nabla_{tt} f(\vect{\zeta},\theta)+ \\
&\delta\,(\nabla_t f(\vect{x}_k,t_k)-|\nabla_t f(\vect{x}_k,t_k)|)+\\
&\frac{\delta^2\nabla_t f(\vect{x}_{k},t_k)^2}{2\|\nabla_{\vect{x}} f(\vect{x}_{k},t_k)\|^4}
\nabla_{\vect{x}} f(\vect{x}_{k},t_k)^\top \nabla_{\vect{xx}}f(\vect{\zeta},\theta)\nabla_{\vect{x}} f(\vect{x}_{k},t_k)\\&-\frac{\delta^2\nabla_{t} f(\vect{x}_{k},t_k)}{\| \nabla_{\vect{x}} f(\vect{x}_{k},t_k)\|^2}  \nabla_{t\vect{x}}f (\vect{\zeta},\theta)^\top \nabla_{\vect{x}} f(\vect{x}_{k},t_k),
\end{align*}
given Assumption~\ref{asm:bound_dfstar} and applying Cauchy–Schwarz inequality results in 
\begin{align*}
    &|f(\vect{x}^-_{k+1},t_{k+1})- f(\vect{x}_{k},t_{k})| \leq \frac{\delta^2}{2}K_3+ 2\delta K_1 \\
&+\frac{\delta^2 K_1^2 M }{2\|\nabla_{\vect{x}} f(\vect{x}_{k},t_k)\|^2}
+\frac{\delta^2 K_1 K_2}{\| \nabla_{\vect{x}} f(\vect{x}_{k},t_k)\|},
\end{align*}

% which results in
% \begin{align*}
% &|f(\vect{x}^-_{k+1},t_{k+1})- f(\vect{x}_{k},t_{k})|\leq \frac{\delta^2}{2}\|\nabla_{tt} f(\vect{\zeta},\theta)\|+\\
% &\frac{\delta^2}{2\|\nabla_{\vect{x}}f(\vect{x}_{k},t_k)\|^2}
% \|\nabla_t f(\vect{x}_{k},t_k)^2\| \|\nabla_{\vect{xx}}f(\vect{\zeta},\theta)\|\\&+\frac{\delta^2}{\| \nabla_{\vect{x}} f(\vect{x}_{k},t_k)\|}  \|\nabla_{t\vect{x}}f (\vect{\zeta},\theta)\| \|\nabla_t f(\vect{x}_{k},t_k)\|,
% \end{align*}

 \begin{figure}[t]
  \centering
    \includegraphics[scale=0.4]{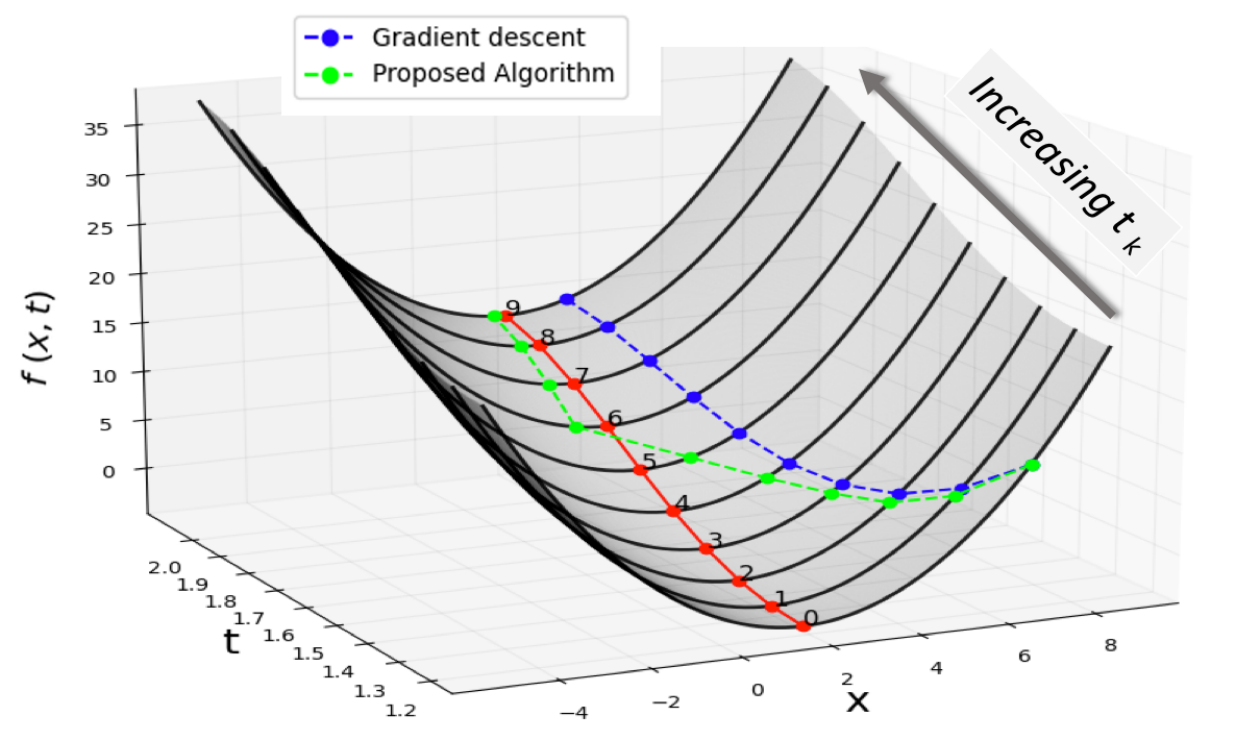}
      \caption{\small{The difference between the trajectories of the proposed Algorithm 1 versus the gradient descent algorithm shown in time interval $t\in[1.2, 2.1]$ when the cost is the time-varying cost shown in Fig.~\ref{fig:illustration_page1}.} }\label{fig:illustration}
\end{figure}

since $1/\|\nabla_{\vect{x}}f(\vect{x}_k,t_k)\|\leq 1/\eps$ results in $|f(\vect{x}^-_{k+1},t_{k+1})- f(\vect{x}_{k},t_{k})|\leq \gamma\delta^2$.
Therefore, taking into account~\eqref{eq::grad_less_eps}, we can conclude that the prediction steps $2$ to $5$ of Algorithm~\ref{Alg::1} lead to
\begin{align} \label{eq::bound_alg1}
|f(\vect{x}^-_{k+1},t_{k+1})- f(\vect{x}_{k},t_{k})|\leq \max(\gamma\delta^2,\mu\delta)%\frac{\eps^2}{2m})
\end{align}

where $\gamma$ is as given in the statement. For the update step, given Assumption~\ref{asm:str_convexity} in light of~\eqref{eq::convex_lipsc}, we have 
\begin{align*}
&\nabla_{\vect{x}}  f(\vect{x}^-_{k+1},t_{k+1})^\top(\vect{x}_{k+1}-\vect{x}^-_{k+1})+\frac{m}{2}\|\vect{x}_{k+1}-\vect{x}^-_{k+1}\|^2\\&\leq f(\vect{x}_{k+1},t_{k+1})-f(\vect{x}^-_{k+1},t_{k+1})\leq\\&\nabla_{\vect{x}}  f(\vect{x}^-_{k+1},t_{k+1})^\top(\vect{x}_{k+1}-\vect{x}^-_{k+1})+\frac{M}{2}\|\vect{x}_{k+1}-\vect{x}^-_{k+1}\|^2,
\end{align*}
 which along with update step $7$ of the Algorithm~\ref{Alg::1}~gives
\begin{align*}
  &-\alpha (1-\frac{m}{2}\alpha)\|\nabla_{\vect{x}}  f(\vect{x}^-_{k+1},t_{k+1})\|^2\leq f(\vect{x}_{k+1},t_{k+1})-\\
&\quad f(\vect{x}^-_{k+1},t_{k+1})\leq -\alpha (1-\frac{M}{2}\alpha)\|\nabla_{\vect{x}}   f(\vect{x}^-_{k+1},t_{k+1})\|^2.
\end{align*}
%To guarantee successive descent we set $0<\alpha\leq \frac{2}{M}$. 
Next, note that given Assumption~\ref{asm:str_convexity} and in light of~\eqref{eq::dist_optim_cost}, we have 
 \begin{align*}
  \!\! -2M (&f(\vect{x}^-_{k+1},t_{k+1})-f^\star(t_{k+1}))\!\leq\!\!-\|\nabla_{\vect{x}}  f(\vect{x}^-_{k+1},t_{k+1})\|^2\\&  \!\leq -2m (f(\vect{x}^-_{k+1},t_{k+1})-f^\star(t_{k+1})).
\end{align*}
Let us write $ f(\vect{x}_{k+1},t_{k+1})-f^\star(t_{k+1})
   \!=\!f(\vect{x}_{k+1},t_{k+1})-f(\vect{x}^-_{k+1},t_{k+1})+(f(\vect{x}^-_{k+1},t_{k+1})-f^\star(t_{k+1})),$
which together with the bounds we already derived for the update step leads to
\begin{align}\label{eq::bound_int}
   & f(\vect{x}_{k+1},t_{k+1})\!-\!f^\star(t_{k+1})\leq\nonumber\\
    &\quad \quad\quad  (1-2\kappa  \alpha m)\,(f(\vect{x}^-_{k+1},t_{k+1})-f^\star(t_{k+1})),
\end{align}
where $\kappa=(1-\alpha M/2)$. Here, given that $0<m\leq M$, we used the fact that for $0<\alpha\leq \frac{1}{2M}$ we have $0<2\alpha m(1-\alpha M/2)<2\alpha M(1-\alpha m/2)\leq 1$.
 On the other hand, by taking into account~\eqref{eq::bound_alg1} and \eqref{eq::bound_optimal_trajectory} we can write
\begin{align*}
&f(\vect{x}^-_{k+1},t_{k+1})-f^*(t_{k+1})=f(\vect{x}^-_{k+1},t_{k+1})-f(\vect{x}_{k},t_{k})\\&+f(\vect{x}_{k},t_{k})-f^*(t_{k})+f^*(t_{k})-f^*(t_{k+1})\leq \\
&\max(\gamma\delta^2,\mu\delta)+f(\vect{x}_{k},t_{k})-f^*(t_{k})+ \psi.
\end{align*}
Then, from~\eqref{eq::bound_int} we can obtain  
 \begin{align*}
&f(\vect{x}_{k+1},t_{k+1})-f^*(t_{k+1}) \leq(1-2\kappa \alpha m )\max(\gamma\delta^2,\mu\delta)%\frac{\eps^2}{2m})
\\&+
(1-2\kappa \alpha m )(f(\vect{x}_{k},t_{k})-f^*(t_{k}))+(1-2\kappa \alpha m )\psi.
\end{align*}
Subsequently, since $0<1-2\kappa \alpha m\leq 1$ we can obtain~\eqref{eq::bound_on_solu_1}, completing the proof.
\end{proof}

%remark on the bound in~\eqref{eq::bound_on_solu_1}
\begin{rem}[Ultimate tracking bound of Algorithm 1]\label{rem::Algorithm1}
{
\rm
The tracking bound of Algorithm~1 is given by~\eqref{eq::bound_on_solu_1}.  As $k\to\infty$ we have $(1-2\kappa\alpha m)^k\to0$ in~\eqref{eq::bound_on_solu_1}. 
Thus, the effect of initialization error $f(\vect{x}_{0},t_{0})-f^*(t_{0})$ vanishing with time. Moreover, the ultimate bound on $f(\vect{x}_{k+1},t_{k+1})-f^*(t_{k+1})$ as $k\to\infty$ is $\frac{\psi}{4\kappa^2\alpha^2 m }+\frac{\max(\gamma\delta^2,\mu\delta)}{4\kappa^2\alpha^2 m^2}$. Thus the optimal value of $\eps$ corresponding to the lowest bound in~\eqref{eq::bound_on_solu_1} is obtained as the solution of  $\gamma\delta^2=\mu\delta$,
%\frac{\eps^2}{2m}$ 
which can be calculated numerically.
}

\end{rem}
Algorithm~\ref{Alg::1} requires explicit knowledge of $\nabla_{t}f(\vect{x}_{k},t_k)$ which may not be available for costs constructed from streaming data. For such problems we propose Algorithm~\ref{Alg::2}, which follows the same prediction-correction structure of Algorithm~\ref{Alg::2} but uses an approximation for $\nabla_t f(\vect{x}_k,t_k)$. In Algorithm~\ref{Alg::2}~%, designed for applications that explicit value of  $\nabla_{t} f(\vect{x}_{k},t_k)$ is not available,  we approximate $\nabla_{t} f(\vect{x}_{k},t_k)$ by
\begin{align}
 \nabla_{t} f(\vect{x}_{k},t_k)\approx\frac{f(\vect{x}_{k},t_k)-f(\vect{x}_{k},t_{k-1})}{t_k-t_{k-1}}.
\end{align}
Higher-order differences can also be used to construct a better approximation of $f(\vect{x}_{k},t_k)$, but at the expense of higher computation and storage costs.
In both Algorithm~\ref{Alg::1} and Algorithm~\ref{Alg::2}, $\eps\in\real_{> 0}$ is an arbitrary chosen parameter that, as we show below, can be tuned to achieve a desired level of tracking accuracy. Next, theorem, explains the convergence guarantee of Algorithm~\ref{Alg::2}.

\begin{algorithm}[t]
\caption{$\eps-$exact Time-Varying Optimization with $\nabla_t f(\vect{x}_k,t_k)$ Approximation }
{%\scriptsize
\begin{algorithmic}[1]
\State {$\mathbf{\textbf{Initialization}}\,\, \vect{x}_0\in\real^n,\,\,\, \eps\in\real_{>0}, f(\vect{x}_0,t_0)\in\real$,} 
\If{$\|\nabla_{\vect{x}}f(\vect{x}_k,t_k)\|\geq \eps$}
\State $\vect{x}^-_{k+1}=\vect{x}_{k}-\frac{ |f(\vect{x}_k,t_{k})-f(\vect{x}_k,t_{k-1})|}{\|\nabla_{\vect{x}}f(\vect{x}_k,t_k)\|^2}\nabla_{\vect{x}}f(\vect{x}_k,t_k)$
\Else
\State{ \textbf{Set} $\vect{x}^-_{k+1}=\vect{x}_{k}$}
\EndIf
\State \textbf{Update} $f(\vect{x}^-_{k+1},t_{k+1})$
\State $\vect{x}_{k+1}=\vect{x}^{-}_{k+1}-\alpha\nabla_{\vect{x}}f(\vect{x}^{-}_{k+1},t_{k+1})$ 
.
\end{algorithmic}}
\label{Alg::2}
\end{algorithm}

\vspace{0.1in}
\begin{thm}[Convergence analysis of the algorithm~\ref{Alg::2}]\label{thm::main2}
{\rm 
Let %$\vect{x}^\star(k)$ be defined as in~\eqref{eq::opt}  and let 
Assumptions~\ref{asm:str_convexity} and~\ref{asm:bound_dfstar}  hold. Then, the Algorithm~\ref{Alg::2} converges to the neighborhood of the optimum solution of~\ref{eq::opt} with the following upper bound
\begin{align}\label{eq::bound_on_solu_2}
&\!\!\!\!f(\vect{x}_{k+1},t_{k+1})-f^*(t_{k+1}) \!\leq\!\frac{\big(1-(1-2\kappa\alpha m)^k\big)}{4\kappa^2\alpha^2 m}\psi\nonumber\\
&\qquad\quad\qquad
+\frac{\big(1-(1-2\kappa\alpha m)^k\big)}{4\kappa^2\alpha^2 m^2}\max(\gamma'\delta^2,
%\frac{\eps^2}{2m}
\mu\delta)\nonumber\\
&\qquad\qquad\quad+
(1-2\kappa\alpha m)^k(f(\vect{x}_{0},t_{0})-f^\star(t_{0})),
\end{align}
where $\psi = \delta (K_1 + \frac{\delta}{2}K_3) + \frac{K^2_2 \delta^2}{2m} (\frac{M\delta}{m} + 2)$, $\gamma'=K_3 + \frac{2K_1}{\delta} + \frac{K_1^2 M}{\epsilon^2} + \frac{K_2(K_1 + \frac{\delta}{2} K_3)}{\epsilon} + \frac{\delta^2 K_3^2 M}{4\epsilon^2}$, 
$\mu=K_1+\frac{\delta}{2}K_3$ and 
$\kappa=(1-\alpha M/2)$,
%$\kappa=\min\big((1-\alpha M/2),(1-\alpha m/2)\big)\in\real_{>0}$
provided that $0<\alpha\leq \frac{1}{2M}$.
}
\end{thm}

\begin{proof}
First note that if $\|\nabla_{\vect{x}}f(\vect{x}_k,t_k)\|< \eps$ Algorithm~\ref{Alg::2} also results in~\eqref{eq::grad_less_eps}. If $\|\nabla_{\vect{x}}f(\vect{x}_k,t_k)\|\geq  \eps$ we proceed as follows. Note that
\begin{align}\label{eq::taylor_proof}
    &f(\vect{x}_k,t_{k-1})=f(\vect{x}_k,t_{k})-\nabla_t f(\vect{x}_k,t_k)(t_{k}-t_{k-1})\nonumber\\&+\frac{1}{2}\nabla_{tt} f(\vect{x}_k,\nu)(t_{k}-t_{k-1})^2,~\nu\in[t_{k-1},t_k).
\end{align}
 Considering~\eqref{eq::taylor_proof}, if $\|\nabla_{\vect{x}}f(\vect{x}_k,t_k)\|\geq  \eps$, the prediction step of Algorithm~\ref{Alg::2} (step 3) results in
% \Mohammadreza{Proof was wrong; it missed the coefficient for $\nabla_{t\vect{x}}f$ and $\frac{\delta}{2}\nabla_{tt} f(\vect{x}_{k},\nu))$\\
% please check the following which i wrote}
\begin{align*}
 &f(\vect{x}^-_{k+1},t_{k+1})- f(\vect{x}_{k},t_{k})=\frac{\delta^2}{2}\nabla_{tt} f(\vect{\zeta},\theta) + \delta \nabla_{t}f(\vect{x}_{k},t_k)\\
&- \delta \big|\nabla_{t}f(\vect{x}_{k},t_k) - \frac{\delta}{2}\nabla_{tt} f(\vect{x}_{k},\nu))\big|\\
& +\frac{\delta^2 \big|\nabla_{t}f(\vect{x}_{k},t_k) - \frac{\delta}{2}\nabla_{tt} f(\vect{x}_{k},\nu))\big|^2}{2\|\nabla_{\vect{x}} f(\vect{x}_{k},t_k)\|^4}\\
&\nabla_{\vect{x}} f(\vect{x}_{k},t_k)^\top \nabla_{\vect{xx}}f(\vect{\zeta},\theta)\nabla_{\vect{x}} f(\vect{x}_{k},t_k)\\
&-\frac{\delta^2 |\nabla_t f(\vect{x}_{k},t_k)-\frac{\delta}{2}\nabla_{tt} f(\vect{x}_{k},\nu)|}{\| \nabla_{\vect{x}}f(\vect{x}_{k},t_k)\|^2}  \nabla_{t\vect{x}} f (\vect{\zeta},\theta)^\top \nabla_{\vect{x}} f(\vect{x}_{k},t_k),
\end{align*}
which given Assumption~\ref{asm:bound_dfstar} and applying Cauchy–Schwarz inequality and using the fact that $|a+b|^2 \leq 2|a|^2 + 2|b|^2$, $\forall$ $a,b \in \real$ we have 
\begin{align*}
&|f(\vect{x}^-_{k+1},t_{k+1})- f(\vect{x}_{k},t_{k})| \leq \delta^2K_3 + 2\delta K_1\\
& +\frac{\delta^2 K_1^2 M}{\|\nabla_{\vect{x}} f(\vect{x}_{k},t_k)\|^2} +\frac{\delta^2 K_2 (K_1 + \frac{\delta}{2}K_3)}{\| \nabla_{\vect{x}}f(\vect{x}_{k},t_k)\|} \\
&+\frac{\delta^4 K_3^2 M}{4\|\nabla_{\vect{x}} f(\vect{x}_{k},t_k)\|^2},
\end{align*}

% \begin{align*}
% &f(\vect{x}^-_{k+1},t_{k+1})- f(\vect{x}_{k},t_{k})=\frac{\delta^2}{2}\nabla_{tt} f(\vect{\zeta},\theta)+\\
% &\frac{\delta^2\nabla_t f(\vect{x}_{k},t_k)^2}{\|\nabla_{\vect{x}} f(\vect{x}_{k},t_k)\|^4}
% \nabla_{\vect{x}} f(\vect{x}_{k},t_k)^\top \nabla_{\vect{xx}}f(\vect{\zeta},\theta)\nabla_{\vect{x}} f(\vect{x}_{k},t_k)\\
% &-\frac{\delta^2}{\| \nabla_{\vect{x}}f(\vect{x}_{k},t_k)\|^2}  \nabla_{t\vect{x}}f (\vect{\zeta},\theta) (\nabla_t f(\vect{x}_{k},t_k)+\frac{\delta^2}{2}\nabla_{tt} f(\vect{x}_{k},\nu))\\&
% \frac{\delta^4\nabla_{tt} f(\vect{x}_{k},\nu)^2}{4\|\nabla_{\vect{x}} f(\vect{x}_{k},t_k)\|^4}
% \nabla_{\vect{x}} f(\vect{x}_{k},t_k)^\top \nabla_{\vect{xx}}f(\vect{\zeta},\theta)\nabla_{\vect{x}} f(\vect{x}_{k},t_k),
% \end{align*}

% which results in
% \begin{align*}
% &|f(\vect{x}^-_{k+1},t_{k+1})- f(\vect{x}_{k},t_{k})|\leq \frac{\delta^2}{2}\|\nabla_{tt} f(\vect{\zeta},\theta)\|+\\
% &\frac{\delta^2}{\|\nabla_{\vect{x}}f(\vect{x}_{k},t_k)\|^2}
% \|\nabla_t f(\vect{x}_{k},t_k)^2\| \|\nabla_{\vect{xx}}f(\vect{\zeta},\theta)\|\\&+\frac{\delta^2}{\| \nabla_{\vect{x}} f(\vect{x}_{k},t_k)\|^2}  \|\nabla_{t\vect{x}}f (\vect{\zeta},\theta)\| \big(\|\nabla_t f(\vect{x}_{k},t_k)\|\\&+\frac{\delta^2}{2}\|\nabla_{tt} f(\vect{x}_{k},t_k)\|\big),
% \end{align*}

because $1/\|\nabla_{\vect{x}}f(\vect{x}_k,t_k)\|\leq 1/\eps$ results in $|f(\vect{x}^-_{k+1},t_{k+1})- f(\vect{x}_{k},t_{k})|\leq \gamma'\delta^2$.
Therefore, %given Assumption%~\ref{asm:str_convexity} and
~\ref{asm:bound_dfstar}, taking into account~\eqref{eq::grad_less_eps}, we can conclude that the prediction steps $2$ to $5$ of Algorithm~\ref{Alg::2} lead to
\begin{align}\label{eq::bound_alg11}
|f(\vect{x}^-_{k+1},t_{k+1})- f(\vect{x}_{k},t_{k})|\leq \max(\gamma'\delta^2,\mu\delta),%\frac{\eps^2}{2m})
\end{align}
where $\gamma'$ is given in the statement. For the update step, we follow the similar approach as the proof Algorithm~\ref{Alg::1} to arrive as same inequality relation~\eqref{eq::bound_int}.
%write
%\begin{align}\label{eq::correction_bound1}
%    &f(\vect{x}_{k+1},t_{k+1})\!-\!f^\star(t_{k+1}) \leq\\&\qquad\qquad (1\!-\!2\kappa\alpha m)(f(\vect{x}^-_{k+1},t_{k+1})\!-\!f^\star(t_{k+1}))\nonumber.
%\end{align}
On the other hand, we have $f(\vect{x}^-_{k+1},t_{k+1})-f^*(t_{k+1})=f(\vect{x}^-_{k+1},t_{k+1})-f(\vect{x}_{k},t_{k})+f(\vect{x}_{k},t_{k})-f^*(t_{k})+f^*(t_{k})-f^*(t_{k+1})$, which along with 
invoking~\eqref{eq::bound_alg11}, \eqref{eq::bound_optimal_trajectory} and~\eqref{eq::bound_int} results in
\begin{align*}
&f(\vect{x}_{k+1},t_{k+1})-f^*(t_{k+1})\leq(1-2\kappa\alpha m)\max(\gamma'\delta^2,\mu\delta)%\frac{\eps^2}{2m})
\\&+
(1-2\kappa\alpha m)(f(\vect{x}_{k},t_{k})-f^*(t_{k}))+\\&(1-2\kappa\alpha m)\psi,
\end{align*}
and consequently~\eqref{eq::bound_on_solu_2},
which concludes our proof.
\end{proof}

Based on Theorem~\ref{thm::main2}, a similar statement to Remark~\ref{rem::Algorithm1}
can be made about the ultimate tracking bound of Algorithm~2. Notice also that the tracking error of Algorithm~\ref{Alg::2}, as one can expect based one Algorithm~\ref{Alg::2}'s use of an estimate for $\nabla_{t} f(\vect{x}_{k},t_k)$, can be larger than Algorithm~1's because $\gamma'\geq \gamma$.

\section{Numerical Example}\label{sec::Num_ex}
In this section, we demonstrate the performance of Algorithm \ref{Alg::1} and Algorithm \ref{Alg::2} using three different examples. In the first example we assume the time derivative of cost are explicitly available. In the next two examples, which include solving a Model Predictive Control (MPC) as a  time-varying convex optimization problem and a learning problem with streaming date, the explicit knowledge about the time derivative of the cost is not available so turn to solving these problems using Algorithm~\ref{Alg::2}. 

 \begin{figure}[t]
  \centering
    \includegraphics[scale=0.4]{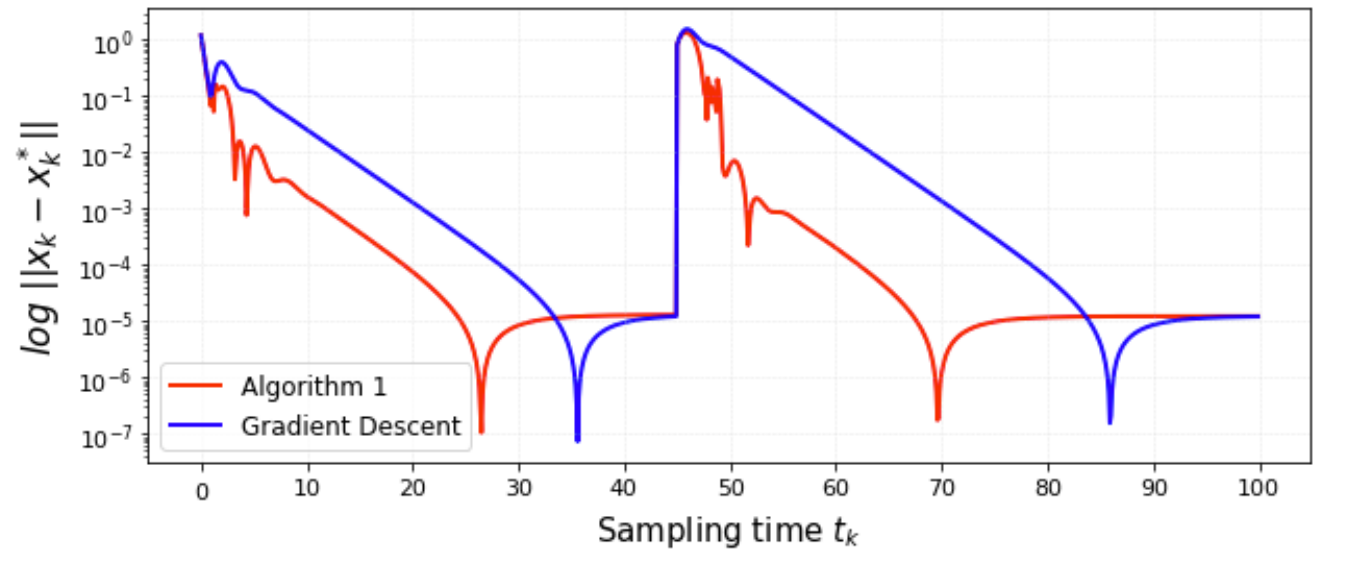}
      \caption{log error of the performance of Algorithm \ref{Alg::1} versus gradient descent algorithm with respect to the sampling time $t_k$.} \label{num1_1}
\end{figure}
\subsection{A case of time-varying cost function whose time derivatives are available explicitly}
Consider a time-varying cost function, i.e., we have 
\begin{align}
&f(\vect{x}, t)= (x_1+x_2-0.01)^2 + (1+e^{-(t - \tau)})x_2^2 \nonumber \\
 &+ e^{-(t - \tau)} x_1 \cdot \sin (2t), \nonumber \\
&\text{where} \quad \tau = 
\left\{
    \begin{array}{lr}
    0, & \text{if } t < 45\\
    45, & \text{if } t \geq 45
    \end{array}
\right\} ~~ \text{and}~ \vect{x} = [x_1, x_2]^\top.
\end{align}
% \begin{align}
% f(\vect{x}, t) = 
% \left\{
%     \begin{array}{lr}
%         (x_1+x_2-0.01)^2 + (1+e^{- t})\cdot x_2^2 + e^{- t} \cdot x_1 \cdot \sin (2t), & \text{if } t < 45\\
%         (x_1+x_2-0.01)^2 + (1+e^{- t})\cdot x_2^2 + e^{- t} \cdot x_1 \cdot \sin (2t), & \text{if } t \geq 45
%     \end{array}
% \right\} = yz
% \end{align}
% \begin{align}
% \label{eq::num1}
%     \min_{x\in \real}{f(x,t)} &= 0.5 \big (x - \kappa \cos(\frac{\pi}{10} t) \big)^2 + g \log \big(1 + \exp(\mu x)\big).
% \end{align}
Here, the purpose of the variable $\tau$ is to produce a jump within the function at the time instant $t = 45$ to observer the response of the algorithm in relation to abrupt changes in the underlying problem. Moreover, we can easily establish the optimal point of the above time-varying cost function as time goes to infinity is $\vect{x^\star} = \vect{x}_t^* = [2, -1]^\top$ as $t \to \infty$, i.e., $\vect{x^\star} = \arg\underset{{\vect{x}\in
    \reals^2}}{\min} 
  \,\,\underset{t\to \infty}{\lim} f(\vect{x},t)$. 
  
  As depicted in Figure \ref{num1_1}, initializing both gradient descent and Algorithm \ref{Alg::1} from $\vect{x}_0 = [0.1, 1.2]^\top$, selecting $\alpha = 0.04$, $\delta = 0.1$ and $\eps = 0.03$, for $t < 45$, Algorithm \ref{Alg::1} demonstrates the ability to attain an error level of $10^{-3}$ at $t_k = 10.7$, whereas the gradient descent algorithm achieves the same error threshold at $t_k =24.7$. In other words, gradient descent algorithm requires $140$ more iterations to reach to the error level of $10^{-3}$. The simulation result showcases Algorithm \ref{Alg::1} attains a level of precision with tracking error of $\epsilon$ faster than gradient descent algorithm. Furthermore, for $t \geq 45$, after inducing a jump within the function at $t = 45$, Algorithm \ref{Alg::1} demonstrates the ability to attain an error level of $10^{-3}$ at $t_k = 54$, whereas the gradient descent algorithm achieves the same error threshold at $t_k =71$. In other words, gradient descent algorithm requires $170$ more iterations to reach to the error level of $10^{-3}$. The simulation results vividly indicate that after reaching the error threshold $\epsilon$, Algorithm \ref{Alg::1} behaves similarly to the gradient descent algorithm until both algorithms converge to the steady-state error bound. 
  
  % Similarly, as depicted in Figure \ref{num1_2}, elevating the frequency results in more oscillations during the initial iteration stages, however, Algorithm \ref{Alg::1} reaches to the accuracy level of $10^{-3}$ at time $t_k = 8.1$ while the gradient descent algorithm reaches to this level of accuracy at $t_k = 20$ (i.e., indicating 119 more iterations needed to obtain the same level of accuracy achieved by Algorithm \ref{Alg::1}).
% we consider two sets of experiments. For the first and second case our parameters are $g = 5 * 10^{-4}$, $\kappa = 10$, $\mu = 1 * 10^{-3}$ and $g = 3 * 10^{-3}$, $\kappa = 5$, $\mu = 9 * 10^{-3}$ respectively. 
% \begin{align*}
%     &\nabla_{xt} f(x,t) = \frac{\pi}{10}\sin(\frac{\pi}{10} t) \\ 
%     &\max_{x\in \real, t} \nabla_{xt} f(x,t) = \frac{\pi}{10}\\
%     &\nabla_{xx} f(x,t) =  1 + g \mu ^2  \frac{\exp(\mu x)}{(1 + \exp(\mu x))^2}\\
%     &\max_{x\in \real, t} \nabla_{xx} f(x,t) = 1 + \frac{g \mu^2}{4}\\
%     &\nabla_{t}f(x,t) = \frac{\pi}{10} \kappa \sin(\frac{\pi}{10} t) \big (x - \kappa \cos(\frac{\pi}{10} t) \big)\\
%     &\max_{x\in \real, t} \nabla_{t} f(x,t) = \frac{\pi ^2}{100} \kappa \big( x + \kappa \big )\\
%     &\nabla_{tt}f(x,t) = 
% \end{align*}

% for this cost function, we have access to the explicit knowledge of $\nabla_t f(x,t)$ for all $t\geq 0$.

 \begin{figure}[t]
  \centering
    \includegraphics[scale=0.36]{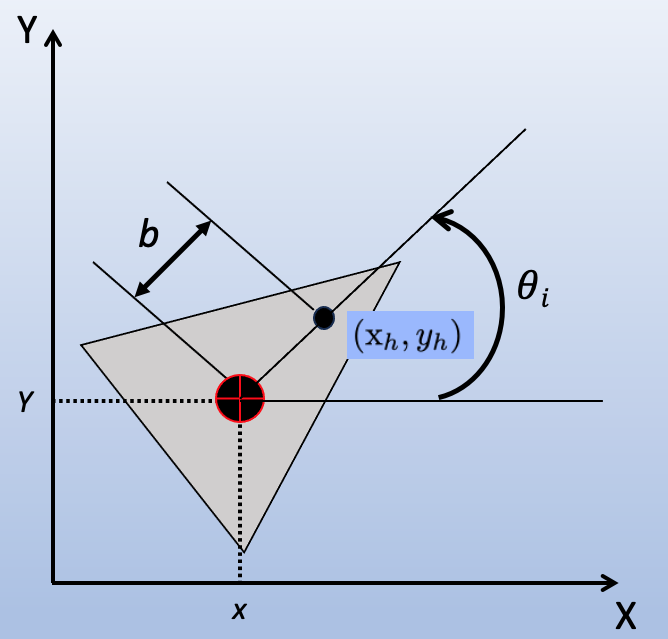}
      \caption{A unicycle robot and its state variables }
      \label{fig:unicycle}
\end{figure}
\subsection{Model Predictive Control}
 MPC consists of solving repeated optimization problems over some fixed moving horizon. These repeated optimization problems can be viewed as an incidence of an optimization problem with streaming/time-varying cost. Let us demonstrate through designing an MPC controller for a unicycle robot, see Fig.~\ref{fig:unicycle}. The objective is to have point $(x_h,y_h)$, which can be for example a camera position on the robot to follow a desired trajectory $$r(t) = \begin{bmatrix}
        r_x(t)\\
    r_y(t)
\end{bmatrix}.$$ Considering the unicycle robot dynamics 
\begin{align}
    \dot X 
    =\begin{bmatrix}
        \dot x\\
        \dot y\\
        \dot \theta
    \end{bmatrix} = 
    \begin{bmatrix}
        v\cos(\theta)\\
        v\sin(\theta)\\
        \omega
    \end{bmatrix}.
    \nonumber
    \end{align}
This desired trajectory is not available a priori and at each time $t$ it is constructed/adjusted from observing the path using the robots camera system for some forward horizon. 

The dynamics governing the motion of our point $(x_h,y_h)$ on the robot is given by (considering the geometry shown in Fig.~\ref{fig:unicycle}) 
    \begin{align}
     X_h &= 
    \begin{bmatrix}
        x_h \\
        y_h
    \end{bmatrix} = 
    \nonumber
    \begin{bmatrix}
        x + b\cos(\theta)\\
        y + b\sin(\theta)
    \end{bmatrix},\\
         \dot X_h &= 
         \nonumber
    \begin{bmatrix}
        \dot x_h \\
        \dot y_h
    \end{bmatrix} = 
    \nonumber
    \begin{bmatrix}
    \nonumber
        \dot x - b\dot \theta\sin(\theta)\\
        \dot y + b\dot \theta \cos(\theta)
    \end{bmatrix} = 
    \begin{bmatrix}
    \nonumber
        u_1\\
        u_2
    \end{bmatrix}.
\end{align}
Discretizing this dynamics we obtain 
\begin{align}
    \begin{bmatrix}
    \nonumber
        x_h(k+1) \\
        y_h(k+1)
    \end{bmatrix} = 
    \begin{bmatrix}
    \nonumber
        x_h(k) + \delta u_1\\
        y_h(k) + \delta u_2
    \end{bmatrix}.
\end{align}
Then, an MPC-based tracking controller can be obtained from solving the following optimization problem at each time step $k$, executing only controller $u(k)$ and repeating the process:
\begin{subequations}\label{eq::MPC}
\begin{align}
    \min_{u_1} J_1(k) &=  \sum_{i=0}^{H_p+H_w-1}(r_x(k+i)-x_h(k+i))^2 \nonumber\\ 
    &+ \frac{1}{\lambda}\sum_{i=0}^{H_u-1}(u_1(k+i))^2,\\
    \min_{u_2} J_2(k) &=  \sum_{i=0}^{H_p+H_w-1}(r_y(k+i)-y_h(k+i))^2 \nonumber\\
    &+ \frac{1}{\lambda}\sum_{i=0}^{H_u-1}(u_2(k+i))^2.
\end{align}
\end{subequations}
Here, $H_p$ and $H_u$ are the length of the prediction horizon and the control horizon, respectively. Also, $H_w$ may be used to alter the control horizon but for the case of simplicity, we put \mbox{$H_w = 0$}. Here, $\lambda$ is the weight factor which is assumed to be constant over the prediction horizon. For our numerical example we use the following values $\lambda = 0.1$, $H_p = H_u = 10$. Additionally, we initialize the robot from $x(0) = -100$ ~and $y(0) = -100$ ~and setting $u_1(i) = u_2(i) = 1, ~ \forall i \in \{0, 1,...,9\}$ when $k=0$. 
 \begin{figure}[t]
  \centering
    \includegraphics[scale=0.38]{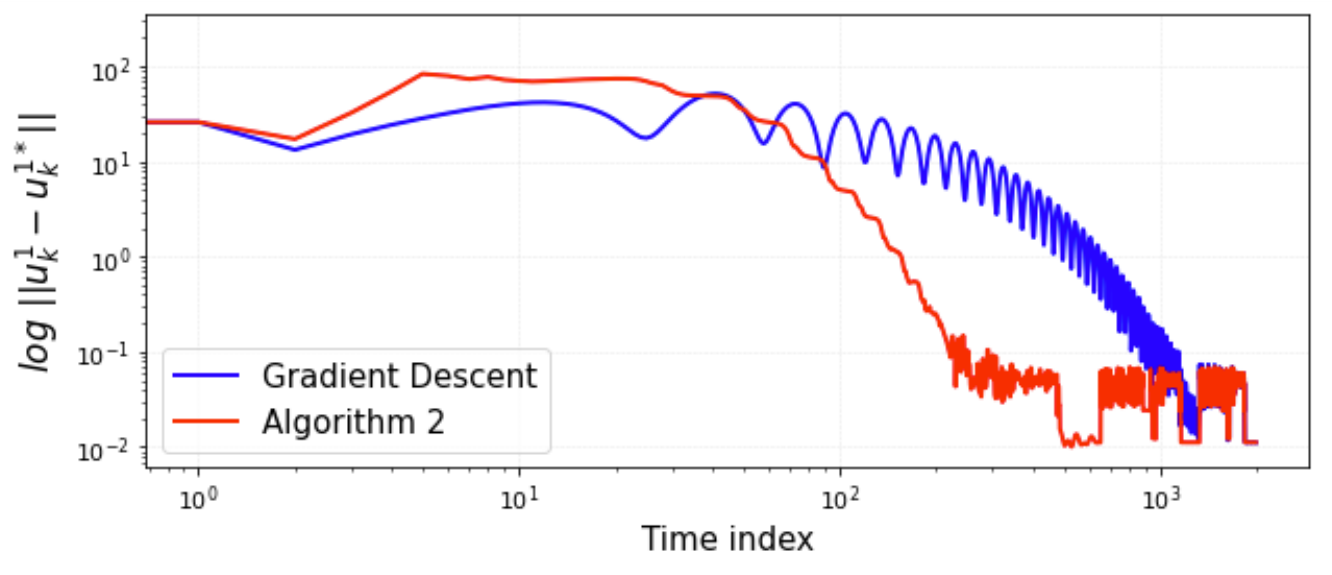}
    \\
     \includegraphics[scale=0.38]{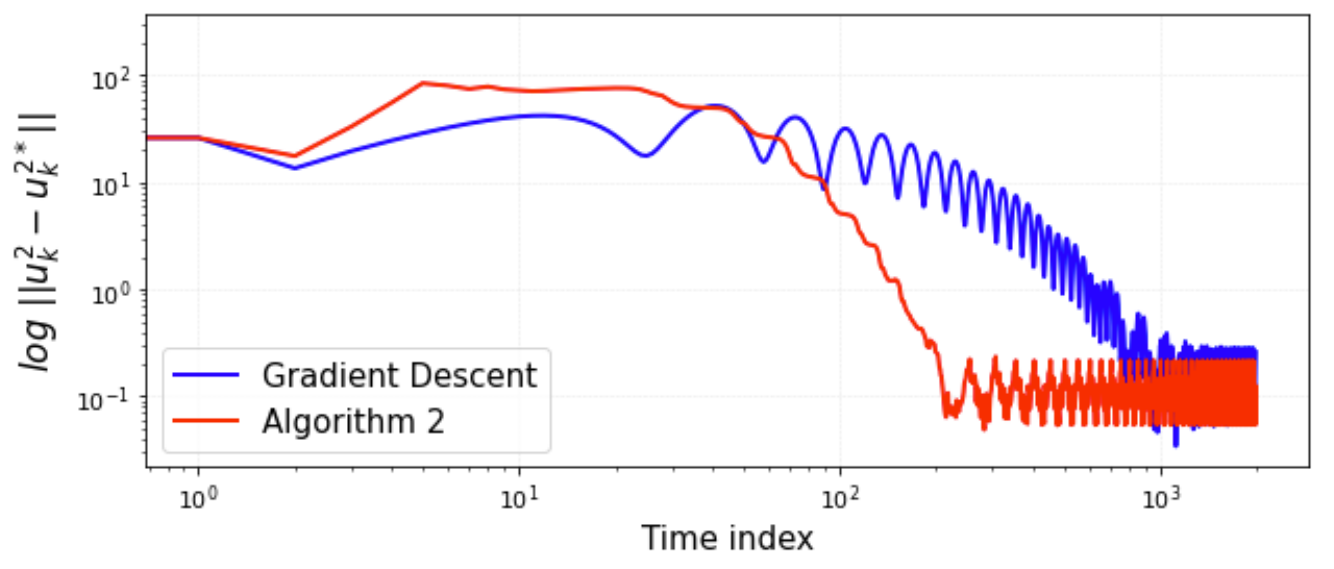}
      \caption{\small{Log error of the controller $u_1$ and $u_2$ with respect to time under executing Algorithm \ref{Alg::2} and gradient descent algorithm. Here $u_k^{l\star}$, $l\in{1,2}$ is the optimal control value which given that the costs are quadratic can be obtained analytically by assuming enough computational power.}}\label{num2_1}
\end{figure}

One of the main challenges with the MPC control is the cost of solving the optimization problems of the form~\eqref{eq::MPC} repeatedly with for example gradient descent algorithm which converges asymptotically at each time step. By looking at the optimization problem~\eqref{eq::MPC} as an instantiation of a time varying cost at time $t_k$ we solve the MPC problem with Algorithm~\ref{Alg::2} whose parameters are set to $\alpha = 0.01$, $\eps = 10^{-1}$ and $\delta = 0.1$ for the sampling timestep. Note here that since we do not have explicit knowledge of the time derivative of the cost, we cannot solve this problem with Algorithm~\ref{Alg::1}. The results are shown in Fig.~\ref{num2_1} which shows that Algorithm~\ref{Alg::1} achieves a better tracking error than gradient descent algorithm. Meaning that we can take less steps to get to a certain tracking error threshold than the gradient descent algorithm, resulting in reducing computation cost.

\subsection{Learning for streaming data}
As a final demonstrative example we demonstrate the performance of the  Algorithm~\ref{Alg::2} in comparison to gradient descent and other first-order static algorithms for a learning problem. This numerical example is taken from the first numerical example of~\cite{ED-AS-SB-LM:19}. This is an example where the exact value of $\nabla_t f(\vect{x}_k,t_k)$ may not be available in practice. Suppose that data points arrive sequentially
at intervals of $\tau>0$
where $\tau$ can be selected as the inter-arrival time of data. The goal is to find the optimal solution of the regression problem at time $t \in T$ based on data $\vect{Z}(t) = \{z(\tau) , \tau\in W_t\}$ where $W_t$ is a sliding window. The problem can be formulated as
\begin{align*} \label{eq::opt}
  \vect{x}^{\star}(t) &= \arg\underset{{\vect{x}\in
    \reals^n}}{\min} 
  \,\,f(\vect{x},\vect{Z}(t)),\quad t\in\real_{>0},
\end{align*}
where $f$ is the quadratic least square cost.
Here we consider an example of a $50$ dimensional time-varying least-squares problem, defined using a sliding window of $50$ data points, for 950 time points. Two big jumps in the solution near time indices $250$ and $550$ are generated by design. Figure~\ref{fig:example-1} shows that our proposed Algorithm~\ref{Alg::2} with $\epsilon=0.01$ outperforms all the known first-order algorithms for static optimization which was simulated in~\cite{ED-AS-SB-LM:19}, even the accelerated algorithms. This can be attributed to the use of implicit knowledge of $\nabla_t f$ in our algorithm, which gives it an anticipatory mechanism about the changes of the cost with time.
%the performance of different accelerated algorithms and the proposed Algorithm~\ref{Alg::2} \blue{with $\epsilon=0.01$}
Figure~\ref{fig:example-1} is in the semi-logarithmic scale over time.
Interestingly, the gradient descent algorithm converges faster than the well-known accelerated algorithms for the time-varying optimization problem. In addition, the proposed algorithm, as shown, outperforms all common optimization algorithms in terms of convergence.

\begin{comment}
 \begin{figure}[t]
  \centering
   % \includegraphics[scale=0.53]{Fig/time_varying2.eps}
    \includegraphics[scale=0.27]{Fig/mpc2.pdf}
      \caption{log error of the controller $u_2$ with respect to time under executing Algorithm \ref{Alg::2} and gradient descent algorithm.}\label{num2_2}
\end{figure}
\end{comment}

 \begin{figure}[t]
  \centering
    \includegraphics[scale=0.22]{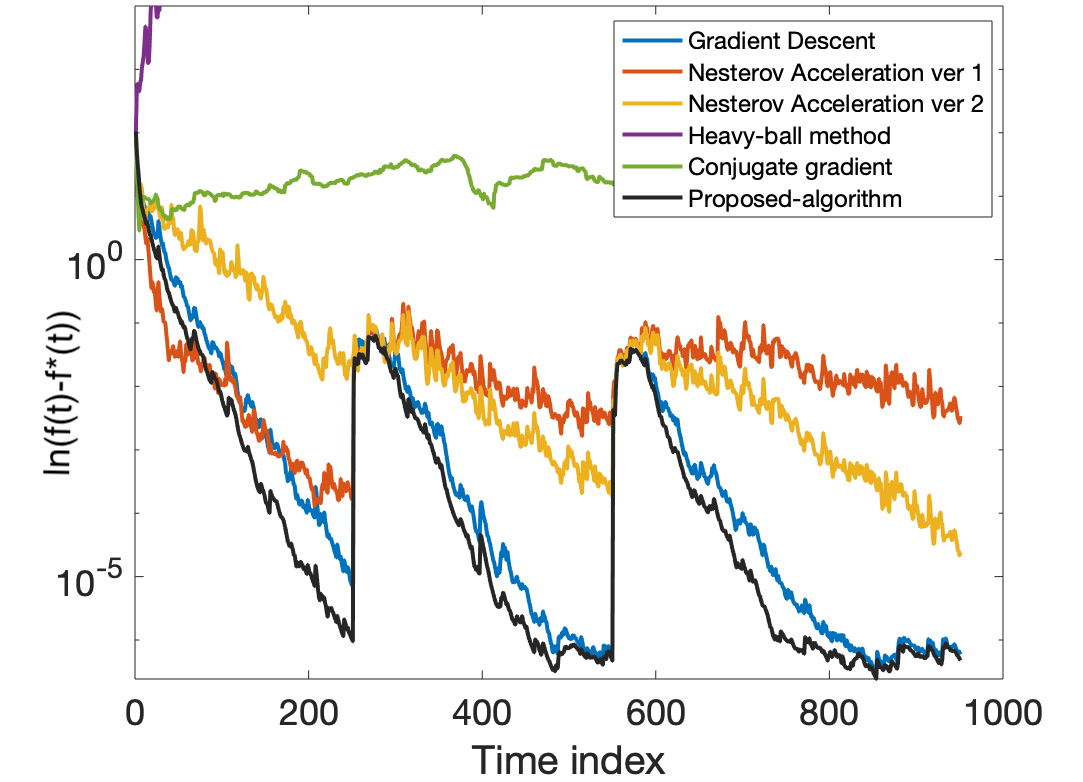}
      \caption{The performance of different algorithms on tracking the optimal objective over an example of a $50$ dimensional time-varying least-squares problem, defined using a sliding window of $50$ data points, for $950$ time points. Two big jumps in the solution near time indices $250$ and $550$ are by design.
Nesterov ver. 1 does not use knowledge of strong convexity, while ver. 2 does. The non-linear conjugate gradient exploits
the quadratic objective to have an exact line-search.}\label{fig:example-1}
\end{figure}

%The example is taken from \cite{FZ-DV-AC-GP-LS:11}.  We consider a
%ring network of $50$ agents where agents can communicate only to their
%left and right neighbors. The local objective functions are randomly
%generated as
%$$f^i(x)=c_ie^{a_ix} +d_ie^{-b_ix},~~ i=1,\cdots,N$$
%where $a_i,b_i \sim\mathcal{U} [0,0.2]$ and $c_i,d_i \sim\mathcal{U} [0,1]$, where $\mathcal{U} $ indicates uniform distribution.
 
%  \begin{figure}[t]
%   \centering
%    % \includegraphics[scale=0.53]{Fig/time_varying2.eps}
%     \includegraphics[scale=0.22]{Fig/time_varying2.png}
%       \caption{The performance of different algorithms on tracking the optimal objective over an example of a $50$ dimensional time-varying least-squares problem, defined using a sliding window of $50$ data points, for $950$ time points. Two big jumps in the solution near time indices $250$ and $550$ are by design.
% Nesterov ver. 1 does not use knowledge of strong convexity, while ver. 2 does. The non-linear conjugate gradient exploits
% the quadratic objective to have an exact line-search.}\label{fig:example-1}
% \end{figure}

\section{Conclusion}\label{sec::conclu}
In this paper,  we proposed two algorithms for a class of convex optimization problems where the cost is time-varying. Our solutions can track the optimal trajectory of the minimizer with bounded steady-state error. Our algorithms are executed only by using the cost function's first-order derivatives, making them computationally efficient for optimization with a time-varying cost function and amenable to non-convex optimization problems. We argued by time-varying cost should not solved by gradient descent and how taking into account how the cost function varies with time can lead to first-order optimization algorithms with lower tracking error. We demonstrated the effectiveness of our proposed algorithms through several examples including an MPC problem and a learning task with a streaming data. Future work devotes to expanding the current algorithms for the problems with non-convex cost functions and investigating distributed implementation of our proposed algorithms for in-network problems where agents communicate according to graph topology.

%---------------------
%---------------------

\bibliographystyle{ieeetr}%
\bibliography{bib/alias,bib/Reference} % bib file to
%\bibliography{alias,SMD-add,Main,SM,JC,SSK-add} % bib file to
%\bibitem{IEEEhowto:kopka}
%H.~Kopka and P.~W. Daly, \emph{A Guide to \LaTeX}, 3rd~ed.\hskip 1em plus
%  0.5em minus 0.4em\relax Harlow, England: Addison-Wesley, 1999.

\end{document}